\documentclass[12pt]{article}

\usepackage{amsmath,amssymb,amsthm}

\usepackage{hyperref}

\title{Proper colorings of a graph in linear time using a number of
colors linear in the maximum degree of the graph}
\author{Kritika Bhandari \and Mark Huber}
\begin{document}
\maketitle

\newcommand{\ind}{\mathbb{I}}
\newcommand{\prob}{\mathbb{P}}
\newcommand{\mean}{\mathbb{E}}
\newcommand{\unifdist}{\textsf{Unif}}

\newtheorem{lemma}{Lemma}
\newtheorem{theorem}{Theorem}

\begin{abstract}
  
A new algorithm for exactly sampling from the set of proper colorings of
a graph is presented. This is the first such algorithm that has an
expected running time that is guaranteed to be linear in the size of a 
graph with maximum degree \( \Delta \)
when the number of colors is greater than \( 3.637 \Delta + 1\).

\end{abstract}

\section{Introduction}\label{introduction}

A \emph{proper coloring} of a graph is an assignment of colors from a
set \(C = \{1, \ldots, k\}\) to each node such that every edge of the
graph has the endpoints of the edge assigned different colors. That is,
for a graph \(G = (V, E)\) and color set \(C\), \(x:V \rightarrow C\) is
a proper coloring means that
\((\forall \{i, j \} \in E)(x(i) \neq x(j))\). When a graph \(G\) has a
proper coloring, say that \(G\) is \emph{\(k\)-colorable}.

Such colorings were introduced by Guthrie~\cite{guthrie1880} in the context of the four-colorability problem of planar graphs. Karp~\cite{karp1972} showed that determining whether general graphs are \(k\) colorable for \(k \geq 3\) is an
NP-complete problem. Jerrum~\cite{jerrum1995} showed (again when
\(k \geq 3\)) that counting the number of proper colorings of a graph is
\#P-complete.

Let \(\Delta\) be the maximum degree of the graph. When
\(k \geq \Delta + 1\), there must exist a coloring of the graph,
therefore the existence problem is no longer NP-complete. However, the
associated counting problem remains \#P-complete~\cite{jerrum1995}. So next consider the problem of approximating the number of proper colorings within given
relative error \(\epsilon > 0\).

Jerrum, Valiant, and Vazirani introduced a method for creating
randomized approximation algorithms~\cite{jerrumvaliantvazirani1986} for counting the
number of \emph{self-reducible} combinatorial objects given the ability
to sample at least approximately uniformly from the original problem
together with any reduced problems. Fortunately, proper colorings fall
into this class of self-reducible problems. In~\cite{jerrum1995}, Jerrum showed that when
\(k > 2 \Delta\), a simple Gibbs Markov chain (also known as Glauber
dynamics) with stationary distribution uniform over the set of proper
colorings was rapidly mixing using a coupling argument. Therefore this
Markov chain could be used to approximately uniformly sample from
\(k\)-colorings of a graph with \(n\) nodes using \(\Theta(n \ln(n))\) steps
where each step took \( O(\Delta) \) time.

Vigoda~\cite{vigoda1999} then
introduced a new Markov chain that required \(\Theta(n \ln(n))\) steps
to mix when \(k > (11 / 6)\Delta\), however, it required changing the
colors of \(O(n)\) nodes at each step, and so ran in \(O(n^2 \ln(n))\) time
total. This \(11 / 6\) constant was reduced to \((11 / 6) - \epsilon_0\)
where \(\epsilon_0 > 0\) by Chen et al.~\cite{chen2019}, using the same chain
introduced by Vigoda.

At the same time that approximate sampling methods were being developed,
progress on the development of exact sampling techniques was being made
as well. In 1996, Propp and Wilson introduced a \emph{perfect sampling}
protocol called coupling from the past (CFTP)~\cite{proppwilson1996}. This protocol gave a
general method for constructing algorithms that could generate samples
exactly from the stationary distribution of a Markov chain using a
random number of steps in the chain. The method applied automatically to
monotonic Markov chains, unfortunately the chains used by Jerrum and
Vigoda did not fall into this category.

In 1998, Huber~\cite{huber1998}
introduced the idea of a \emph{bounding chain} that ran alongside the
original Markov chain and which allowed CFTP to be applied to the chain.
This method was applied to the coloring chain used by Jerrum to obtain
exactly uniform samples from the set of proper colorings in time
\(\Theta(n \ln(n))\) when \(k > \Delta(\Delta + 2)\).

This bound on the number of colors that was quadratic in the maximum
degree of the graph was finally broken by Bhandari and Chakraborty~\cite{bhandari2020} in 2020.  They only required
\(k > 3 \Delta\) to get a \(\Theta(n \ln(n))\) expected run time. They
again used the bounding chain method. Jain, Sah, and Sawhney~\cite{jain2021} then improved once
more on the bounding chain method, requiring only
\(k > ((8 / 3) + o(1)) \Delta\) colors to obtain a \(\Theta(n \ln(n))\)
average run time algorithm.

Now consider \emph{linear time} algorithms for generating proper
colorings. No known approximate sampling methods exist that use linear
time. A \emph{randomness recycler} style algorithm for this problem was
given in~\cite{fillhuber2000}, but
required \(\Omega(\Delta^4)\) colors to guarantee that the expected running
time is linear.

A randomness recycler (RR) algorithm is an extension of the acceptance
rejection method introduced
formally by Von Neumann in 1951~\cite{vonneumann1951} and the popping technique of Propp and
Wilson from 1998~\cite{proppwilson1998}. The
idea is to first generate a sample from an
easier distribution than the target. Then accept this draw as a
realization from the target with probability built to make the
distribution of the accepted draw equal to that of the target. In the
basic acceptance rejection method, if the draw is not accepted, it is
thrown entirely away.

In an RR approach, instead of discarding the sample entirely, as much of
the sample is kept (recycled) as possible. Yes, the sample has been
biased to a new distribution by the rejection, but that does not mean
that there is not still plenty of randomness left in the sample.

So an RR algorithm keeps track not only of a state, but also of the set
of unnormalized weights that describe the distribution of that state.
The algorithm starts with a set of weights that are easy to sample from,
and then at each step alters either the state, the weights, or both in a
consistent fashion.

An RR algorithm then runs as follows.

\begin{enumerate}
\def\labelenumi{\arabic{enumi}.}
\item
  Draw a state from an initial set of weights.
\item
  Until the weights are a multiple of the target weights, update either
  the state or the weights. Valid updates are explored in Section 2.
\item
  Output the final state.
\end{enumerate}

As shown in~\cite{fillhuber2000,huber2016}, the distribution of this final state will be exactly that of the target distribution.

The main result of this paper is an implementation of the RR protocol
that yields the first algorithm that samples proper colorings in linear
time when the number of colors is linear in the maximum degree.

\begin{theorem}
There exists a randomness recycler algorithm for generating a sample
that is an exact uniform draw from the proper colorings of a graph with \( n \) nodes and  maximum degree \(\Delta\) that uses (on average) \(\Theta(n)\) steps
each taking \(O(\Delta)\) time to execute and using at most
\(O(\Delta \log(k))\) uniform random bits when \(k > 3.637 \Delta + 1 \).
\end{theorem}

Like all perfect simulation algorithms (and unlike the approximate
sampling algorithms) this algorithm can be run when this bound on \(k\)
is not satisfied. If the algorithm terminates, then the result is an
exact sample. The condition is only needed to establish that the
expected running time will be linear before the algorithm is run.

The rest of the paper is organized as follows. The next section reviews
basic concepts in RR type algorithms and presents the new algorithm for
proper colorings. The last section then analyzes the algorithm and shows
the main theorem.

\section{The algorithm}
\label{the-algorithm}

Since the algorithm uses the RR protocol, the general form of this
method will be discussed first.

\subsection{The Randomness Recycler Protocol}
\label{the-randomness-recycler-protocol}

The term \emph{weights} will be used to refer to any unnormalized
density. For example, consider the set of weights \(w(x)\) which is 1
when \(x \in \Omega\) and 0 otherwise. This can be written using the
indicator function \(w(x) = \mathbb{I}(x \in \Omega)\). This set of
weights encodes the uniform distribution over \(\Omega\).

Say that \(w_1\) is equivalent to \(w_2\) if \(w_1 = \alpha w_2\) for a
positive constant \(\alpha\). Since the weights are unnormalized, two
equivalent weight functions represent the same distribution.

An RR algorithm works by doing one of two steps that transforms the
weights.

\begin{enumerate}
\def\labelenumi{\arabic{enumi})}
\item
  Change the state randomly using a transition function \(p(x, y)\).
  This alters the weights of the state using the standard formula for
  Markov chain steps:\\
  \[
   w(x) \mapsto w'(y) = \sum_{x'} w(x') p(x', y).
   \]
\item
  Change the weight randomly. If the current weight function is
  \(w(x)\), then with probability \(p_{\text{acc}}(x)\) move to weights
  \(w(x) p_{\text{acc}}(x)\), and with probability
  \(1 - p_{\text{acc}}(x)\) move to weights
  \(w(x)(1 - p_{\text{acc}}(x))\). This can be viewed as
  \emph{splitting} the weight function \(w(x)\) into two weight
  functions.

  If \(p_1(x), \ldots, p_\ell(x)\) are nonnegative functions that sum to
  \(1\) for every \(x\), then with probability \(p_i(x)\) move to weight
  \(w(x) p_i(x)\). This can be viewed as splitting the weight function
  into \(\ell\) different possibilities.
\end{enumerate}

To use RR, there are two special weights. Let \(w_{\text{init}}\) be a
weight function from which it is easy to sample. Let \(w_{\text{tar}}\)
be the target weight function.

The algorithm begins with a state drawn from weights
\(w_{\text{init}}\). At each step, either a random alteration or a
splitting step is used. Typically \(p(x)\) is chosen so that
\(w(x)p(x)\) is closer to the target distribution. That means that
\(w(x)(1 - p(x))\) is usually farther away from the distribution.

The algorithm terminates when the weight function matches the target
weight function. The final state will be a draw exactly from the
distribution given by the final weights. See Chapter 8 of~\cite{huber2016}
for a proof of this result. (There the result is shown for general state
spaces, but reduces to the simpler form for finite spaces.)

\subsection{The RR algorithm for proper colorings}
\label{the-rr-algorithm-for-proper-colorings}

To build an RR algorithm for proper colorings, a set of indices is
needed that is broad enough to handle the target distribution, a
starting distribution from which it is easy to sample, and the ability
to index the set of weight functions used arising from both acceptance
and rejection.

In the weight function, each node will either be forbidden to have a
color, frozen at a color, can have the same color as neighboring nodes,
or must have a different color than neighboring nodes.

These properties will be encoded using a state \(x^*\) that assigns to
every node a value in \[
A = \{-k, -k + 1, \ldots, -1, 0, 1, 2, \ldots, k\} \cup \{\bot\}.
\] So \(x^* \in \Omega^* = A^V\).

\begin{enumerate}
\def\labelenumi{\arabic{enumi}.}
\item
  \textbf{Forbidden color} If node \(v\) is not allowed to be assigned
  color \(b\), say that node \(v\) is \emph{forbidden} to have color
  \(b\).

  In the index, set \(x^*(v) = -b\) for \(b \in C\) to indicate that
  node \(v\) is forbidden to have color \(b\).
\item
  \textbf{Frozen at a color} If node \(v\) must be assigned color \(c\),
  say that node \(v\) is \emph{frozen} at color \(c\).

  In the index, set \(x^*(v) = c\) for \(c \in C\) to indicate that node
  \(v\) is frozen at color \(c\).
\item
  \textbf{Ignored} Say that the state \emph{ignores} node \(v\) if the
  weights do not take node \(v\) into consideration when calculating if
  the coloring is proper or not.

  Write \(x^*(v) = 0\) to indicate that \(x(v)\) can equal \(x(w)\) for
  any neighbor of \(v\).
\item
  \textbf{Unrestricted} A node is unrestricted if none of the previous
  three conditions apply. Such a node is not allowed to have the same
  color as a neighbor.

  Let \(x^*(v) = \bot\) if the node acts like a regular node. In this
  case \(x(v)\) might be any element of \(C\), and \(x(v) \neq x(w)\)
  for all the neighbors \(w\) of \(v\).
\end{enumerate}

Recall that for indicator functions
\(\mathbb{I}(p \wedge q) = \mathbb{I}(p) \mathbb{I}(q)\) so multiple
requirements can be enforced through multiplication of their respective
indicator functions.

The factor for weight for a forbidden color node is easy, just set
\(\mathbb{I}(x(v) \neq b)\) if \(x^*(v) = -b\). Frozen color nodes use
\(\mathbb{I}(x(v) = x^*(v))\). The ignored nodes do not contribute to
the weight at all.

Finally, the constraint that endpoints of an edge must be different
colors will be enforced as long as at most one of the endpoints is
frozen and neither endpoint is ignored. For edge \(e = \{v, w \}\), this
can be written as \[
f_e(x, x^*) = \mathbb{I}((x(v) \neq x(w)) \vee (x^*(v) = x(v) = x^*(w) = x(w)) \vee (x^*(v) = 0 \vee x^*(w) = 0)).
\]

To be a valid coloring for an index, the correct nodes must be frozen at
certain colors, the correct nodes must be forbidden certain colors, and
for two nodes connected by an edge, either one or both of the nodes are
to be ignored, or they must be assigned different colors.

Let \(-C = \{-1, -2, \ldots, -k \}\). The weight associated with index
\(x^*\) and state \(x \in C^V\) is \begin{align*}
w_{x^*}(x) &= \left[\prod_{v:x^*(v) \in C} \mathbb{I}(x(v) = x^*(v)) \right]
  \left[\prod_{v:x^*(v) \in -C} \mathbb{I}\left(x(v) \neq -x^*(v)\right) \right]
  \left[ \prod_{e \in E} f_e(x, x^*) \right]
\end{align*}

Since \(w_{x^*}(x) \in \{0, 1\}\) (and there is at least one coloring
with positive weight) this is the weight function for a uniform
distribution. There will be three functions that start with a state
\(x\) drawn from weights \(w_{x^*}(x)\), and return a state \(y\) drawn
from weights \(w_{y^*}(y)\). Inside each of these functions, the weights
might temporarily deviate from uniform, but by the end of the function
the weights will be back to one of our indexed weights.

Because each indexed weight function is the product of indicators, the
overall weight function is just uniform over all states with positive
weight under that index.

A helpful notation to have will be \[
a(x^*) = \{x' \in C^V:w_{x^*}(x') = 1\}
\] so that \(x\) given \( x^* \) is uniform over \(a(x^*)\).

The initial index has \(x_0^*(v) = 0\) for all \(v\). Note \[
w_{x_0^*}(x) = \mathbb{I}(x \in C^V).
\] That is, it is uniform over all colorings, not just proper colorings.
This distribution
is straightforward to sample from, just draw each \(x(v)\) uniformly and
independently from \(C\).

If there is ever a node \(v\) that is forbidden a color, try to remove
that restriction first. Next, if there is a node \(v\) frozen at a
color, try to remove that restriction. Finally, if there is a node \(v\)
that is being ignored, try to remove that restriction.

Continue until the final index \(x^*_{\text{tar}}\) is reached, where
every node has value \(\bot\). This corresponds to the weights \[
w_{x^*_{\text{tar}}}(x) = \prod_{\{i, j\} \in E} \mathbb{I}(x(i) \neq x(j)),
\] which is the weight function for the uniform distribution over proper
colorings.

This algorithm can be written as follows.

\vspace*{10pt}

\textbf{RR step for proper colorings}\((x, x^*)\)

\begin{enumerate}
\def\labelenumi{\arabic{enumi}.}
\item
  If there exists \(v \in V\) such that \(x^*(v) = -b\) where
  \(b \in C\) then output \textbf{Remove forbidden}\((x, x^*, v)\) and
  exit.
\item
  Else if there exists \(v \in V\) such that \(x^*(v) = c\) where
  \(c \in C\) then output \textbf{Remove frozen}\((x, x^*, v)\) and
  exit.
\item
  Else if there exists \(v \in V\) such that \(x^*(v) = 0\), then output
  \textbf{Remove ignored}\((x, x^*, v)\) and exit.
\item
  Else output \((x, x^*)\) and exit.
\end{enumerate}

Given this one step algorithm, the main algorithm runs as follows.

\vspace*{10pt}

\textbf{RR algorithm for proper colorings}\((x, x^*)\)

\begin{enumerate}
\def\labelenumi{\arabic{enumi}.}
\item
  Let \(x^*(v) = 0\) for all \(v \in V\).
\item
  For each \(v \in V\), draw \(x(v)\) uniformly and independently from
  \(C\).
\item
  While \(x^* \neq ( \bot, \ldots, \bot)\) replace \((x, x^*)\) with the
  output from \textbf{RR step for proper colorings}\((x, x^*)\).
\item
  Output \((x, x^*)\).
\end{enumerate}

Therefore the heart of the algorithm is how to remove forbidden, frozen,
and ignored conditions from the index. These steps must maintain an
invariant to work.

Let \begin{align*}
(y_1, y^*_1) &= \boldsymbol{\mathsf{Remove\ forbidden}}(y, y^*, v) \\
(y_2, y^*_2) &= \boldsymbol{\mathsf{Remove\ frozen}}(y, y^*, v) \\
(y_3, y^*_3) &= \boldsymbol{\mathsf{Remove\ ignored}}(y, y^*, v) \\
\end{align*}

The invariant to be maintained is that if there is a single node \(v\)
with \(x^*(v) \in -C\) then all nodes \(w\) with \(x^*(w) \in -C\) must
have \(x^*(v) = x^*(w)\). In other words, there is at most one color
that is being forbidden to some of the nodes at any particular time.

Remark: how the node \(v\) is chosen within each step is immaterial to
the algorithm's operation. This could be random or use a deterministic
method such as numbering the nodes and always picking the node with the
smallest label. Any such method will result in a correct algorithm with
the expected running time given in Theorem 1.

Begin with the last procedure, removing an ignored node.

\subsubsection{Remove ignored}
\label{remove-ignored}

Suppose that \(x^*\) has only ignored or unrestricted nodes, and node
\(v\) is ignored. That is, \(x^*(v) = 0\) and for all \(w \neq v\),
\(x^*(w) \in \{0, \bot \}\).

Let \(r^*\) be the same index as \(x^*\) except that the ignored
restriction at \(v\) has been removed. That is,
\(r^*(V \setminus \{ v \}) = x^*(V \setminus \{ v \})\) and
\(r^*(v) = \bot\). Note that for any state \(x\). \[
w_{r^*}(x) = w_{x^*}(x) \prod_{w:\{v, w\} \in E} \mathbb{I}(x(v) \neq x(w)) \leq w_{x^*}(x).
\]

Therefore, the simplest approach to turning weights indexed by \(x^*\)
to weights indexed by \(r^*\) is to accept the result with probability
\(\prod_{w:\{v, w\} \in E} \mathbb{I}(x(v) \neq x(w))\). This
formulation is equivalent to accepting if and only if there is not a
neighbor of \(v\) that has the same color as \(v\).

If acceptance occurs, all is well. But if rejection occurs, the new
weight is \[
w_{x^*}(x) \left[ 1 - \prod_{w:\{v, w\} \in E} \mathbb{I}(x(v) \neq x(w)) \right]
= w_{x^*}(x) \mathbb{I}\left(\left[\sum_{w:\{v, w\} \in E} \mathbb{I}(x(v) = x(w)) \right] > 0 \right)
\] So the new weight is uniform over colorings where node \(v\) does
have a neighbor of the same color.

This condition is (unfortunately) not representable using the indices.
But what can be done is to search the non-ignored neighbors of the node \(v\)
uniformly at random without replacement until a neighbor \(w\) with
\(x(v) = x(w)\) is found.

The neighbors searched before finding this \(w\) cannot have the color
\(x(v)\), and \(w\) must have the color \(x(v)\). This can be represented
in the index by
forbidding the searching neighbors to have color \(x(v)\), and freezing
node \(w\) at color \(x(v)\).

At this point return the index of node \(v\) to \(0\). Finally, redraw
the color of \(v\) uniformly from \(C\).

So now the only task is to find an index \(s^*\) such that \[
w_{s^*}(x') = \mathbb{I}(x' \in a(x^*) \setminus a(r^*)).
\] Consider what this indicator is saying. It is 1 when the coloring
satisfied the conditions of \(x^*\) but not the conditions of \(r^*\).
This can only happen when \(x(v) = x(w)\) for one of the non-ignored
neighbors \(w\) of \(v\).

To find this neighbor, search through the neighbors of \(v\) uniformly
at random without replacement until the neighbor \(w\) matching color
\(x(v)\) is found. Each of the neighbors searched before finding \(w\)
does not have the color: this can be indexed by forbidding the color to
the node. Finally, \(w\) will have to be frozen at the color. At this
point \(w\) can be frozen at the color.

Putting these ideas together gives the following algorithm.

\textbf{Remove ignored}\((x, x^*, v)\)

\begin{enumerate}
\def\labelenumi{\arabic{enumi}.}
\item
  Let \(r^*(V \setminus \{ v \}) = x^*(V \setminus \{ v \})\) and
  \(r^*(v) = \bot\).
\item
  Search the non-ignored and unfrozen neighbors of \(v\) uniformly at
  random without replacement.
\item
  If no non-ignored and unfrozen neighbor of \(v\) has color \(x(v)\),
  output \((x, r^*)\) and exit.
\item
  Else do the following:

  \begin{enumerate}
  \def\labelenumii{\alph{enumii}.}
  \item
    let \(W = \{w_1, \ldots, w_\ell\}\) be the set of neighbors searched
    uniformly at random without replacement before finding non-ignored
    neighbor \(w\) with \(x(w) = x(v)\).
  \item
    Let
    \(s^*(V \setminus (W \cup \{v, w\})) = x^*(V \setminus (W \cup \{v, w\}))\),
    \(s^*(W) = \{-x(v), \ldots, -x(v)\}\), \(s^*(w) = x(w)\),
    \(s^*(v) = 0\).
  \item
    Let \(s(V \setminus \{ v \}) = x(V \setminus \{ v \})\). Draw
    \(s(v)\) uniformly from \(C\).
  \item
    Output \((s, s^*)\) and exit.
  \end{enumerate}
\end{enumerate}

\begin{lemma}
\textbf{Remove ignored} has the weight function given by the output
index for the output state.
\end{lemma}

\begin{proof}

Let \((y, y^*)\) be the output of \textbf{Remove ignored}\((x, x^*)\).

Case 1: \((y, y^*) = (x, r^*)\). Since the state does not change in this
step, the output weight is just \[
w_{x^*}(x) \left[ \prod_{w:\{v, w\} \in E} \mathbb{I}(x(v) \neq x(w)) \right] = w_{r^*}(x)
\] and the case is complete.

Case 2: Suppose \((y, y^*) = (s, s^*)\). If \(s \notin a(s^*)\), then
either one of \(w_1, \ldots, w_\ell\) has color \(x(v)\) and would have
been assigned the role of \(w\) in the search, or \(w\) does not have
the color \(x(v)\), and would not have been chosen as the neighbor with
color \(x(v)\). In either situation, it is not possible for \(s\) to be
the output, so \(w_{s^*}(s) = 0\) as desired.

Now consider when \(s \in a(s^*)\). Then what would the chance of moving
from state \(x\) to state \(s\) be? While \(s^*\) tells us the set of
states \(\{w_1, \ldots, w_\ell\}\) searched before finding the neighbor with
the conflicting color as well as \(w\), it does not reveal the
\emph{order} in which they were searched. Therefore, the probability of
making this particular search, followed by the particular choice of
\(s(v)\), is \[
\frac{\ell!}{\deg(v)(\deg(v) - 1) \cdot (\deg(v) - \ell + 1)} \frac{1}{\deg(v) - \ell} \cdot \frac{1}{k}.
\] Hence \[
w_{s^*}(s) = w_{x^*}(s) \frac{\ell!}{\deg(v)(\deg(v) - 1) \cdot (\deg(v) - \ell + 1)} \frac{1}{\deg(v) - \ell} \cdot \frac{1}{k}.
\]

In particular, \(w_{s^*}(s)\) is a fixed positive constant when
\(s \in a(s^*)\). Therefore, this set of weights can be renormalized.

Since \(w_{x^*}(s)\) is constant over \(s\) where there exists neighbor
\(w\) of \(v\) with \(x(v) = x(w)\), and the other factors do not depend
on \(s\), the weights \(w_{s^*}(s)\) are constant for every state \(s\)
with positive weight. Hence the output state \(y\) is uniform over
\(a(y^*)\).

\end{proof}

\subsubsection{Remove frozen}
\label{remove-frozen}

This is the easiest of the three steps, since if there are no forbidden
nodes, then a frozen node \(v\) can be converted into an ignored node
where all of the unfrozen and non-ignored neighbors of \(v\) are forbidden to have color
\(x(v)\). The color of \(v\) can then be redrawn uniformly from \(C\),
and the state reindexed to be ignored.

For any node \(v\), let \(N(v)\) denote the set of neighbors of \(v\).

\vspace*{10pt}

\textbf{Remove frozen}\((x, x^*, v)\)

\begin{enumerate}
\def\labelenumi{\arabic{enumi}.}
\item
  Let \(r^*\) be the index where node \(v\) is changed to ignored, and
  neighbors of \(v\) that are unfrozen and non-ignored are changed to be forbidden to
  have color \(x(v)\).
\item
  Let \(r\) be the state which equals \(x\) everywhere except \(v\), and
  \(r(v)\) is chosen uniformly at random from \(C\).
\item
  Output \((r, r^*)\).
\end{enumerate}

\textbf{Remove frozen} has the weight function given by the output index
for the output state.

From \(x^*\), the frozen color of \(v\) is known. Therefore, for a given
state \(r\), there is exactly one state \(x\) that moves to \(r\) with
probability \(1 / k\). That means \[
w_{r^*}(r) = w_{x^*}(x) \frac{1}{k} = \mathbb{I}(r \in a(r^*)) \frac{1}{k}.
\] Therefore \(w_{r^*}(r)\) is equivalent to the uniform weights over
\(a(r^*)\), and the output is correct.

\subsubsection{Remove forbidden}
\label{remove-forbidden}

The most difficult restriction to deal with is when a node \(v\) is
forbidden to have color \(b = -x^*(v)\). In this case, it is necessary
to propose switching the color of \(x(v)\) to \(b\). The chance of doing
so needs to be high enough that a later acceptance/rejection step can
bring all the weights down to be the same.

Before starting, if a neighbor of \(v\) is already frozen at color
\(b\), then the forbidden color of \(b\) condition can just be removed
without changing the set of allowable states at all. That is the easy
case.

For the harder cases, start by defining the following sets and
constants. \begin{align*}
  n_1(x, x^*) &= \text{the number of valid colors for } v \text{ in state } x \text{ under index } x^* \\
  n_2(x^*) &= \text{the number of colors not used by frozen neighbors of } v \\
  n_3(x^*) &= \text{the number of neighbors of } v \text{ that are not frozen nor ignored} \\
d(x^*) &= n_2(x^*) - n_3(x^*) \\
U(x^*) &= \text{the set of neighbors of \( v \) that are neither frozen nor ignored} \\
\end{align*}

Note that \(d(x^*) - 1\) represents a lower bound on the number of
colors that are available to \(v\). That is, \[
n_1(x, x^*) \geq d(x^*) - 1.
\]

Form a new state \(r\) that is the same as state \(x\) everywhere except
(possibly) at node \(v\). With probability \(p_1\), set \(r(v) = b\) and
with probability \(1 - p_1\) leave \(r(v) = x(v)\). Note that
\(n_1(x, x^*) = n_1(r, x^*)\) since it does not depend on the color of
node \(v\).

Let \(r^*\) be the same as \(x^*\) for every node except \(v\), and set
\(r^*(v) = \bot\). Then the weight of \(r\) will be positive for every
state in \(a(r^*)\), but it also gives positive weight to states that
satisfy all the restrictions on \(x^*(V \setminus \{v\})\), but allows
\(v\) and a neighbor to both have color \(b\). Call this set of states
\(a_1\). Then \(a(r^*) \subseteq a_1\), so
\(\mathbb{I}(r \in a(r^*)) \leq \mathbb{I}(r \in a_1)\).

Moreover, the weights will not be uniform. For a given coloring \(r\) if
\(r(v) \neq b\) then there was only a single state \(x'\) that could
have moved to \(r\). But note that for \(r(v) = b\), there are
\(n_1(r, x^*)\) different states \(x' \in a(x^*)\) that could have moved
to \(r\).

Hence the weight after this step will be \begin{align*}
w_1(r) &= \mathbb{I}(r \in a_1) \left[ (1 - p_1) \mathbb{I}(x(v) = r(v)) + p_1 \sum_{x' \in a(x^*):x'(V \setminus \{v\}) = r(V \setminus \{ v \}) } \mathbb{I}(r(v) = b) \right] \\
&= \mathbb{I}(r \in a_1) \left[ (1 - p_1) \mathbb{I}(x(v) = r(v)) + p_1 n_1(r, x^*) \mathbb{I}(r(v) = b) \right] \\
\end{align*}

The next step is to only accept this state if \(r \in a(r^*)\), that is,
accept \(r\) as part of \(a(r^*)\) with probability
\(\mathbb{I}(r \in a(r^*)) / \mathbb{I}(r \in a_1)\). If acceptance
occurs, the new state has weight \[
w_2(r) = \mathbb{I}(r \in a(r^*)) \left[ (1 - p_1) \mathbb{I}(x(v) = r(v)) + p_1 n_1(r, x^*) \mathbb{I}(r(v) = b) \right].
\]

If rejection occurs, then it must be that the color of \(r(v) = b\) and
there was already a neighbor of \(v\) with that color. In this case, do
the same thing as in the ignored removal and uniformly at random without
replacement search the neighbors of \(v\) until the neighbor \(w\) with
color \(b\) is found. Forbid the searched neighbors to have color \(b\)
and freeze \(w\) at color \(b\).

Back to the accepted case. Setting \[
p_1 = 1 / d(x^*),
\] then \((1 - p_1) \leq p_1 n_1(r, x^*)\).

Let \[
p_2(r, x^*) = \mathbb{I}(r(v) \neq b) + \frac{d(x^*) - 1}{n_1(r, x^*)} \mathbb{I}(r(v) = b).
\] Then \begin{align*}
w_2(r) p_2(r, x^*) &= 
 \mathbb{I}(r \in a(r^*)) 
 \left[ \left(1 - \frac{1}{d(x^*)} \right)
    \mathbb{I}(r(v) \neq b) + \frac{d(x^*) - 1}{d(x^*)} \mathbb{I}(r(v) = b) \right] \\
    &= \mathbb{I}(r \in a(r^*)) \frac{d(x^*) - 1}{d(x^*)}.
\end{align*}

Therefore \(w_2(r) p_2(r, x^*) \propto w_{r^*}(r)\). So the weights are
great if acceptance occurs.

Under rejection, the weight becomes \begin{align*}
w_2(r) (1 - p_2(r, x^*)) &= 
 \mathbb{I}(r \in a(r^*)) 
 \left[ 0 \cdot \mathbb{I}(r(v) \neq b) + \frac{n_1(r, x^*) - (d(x^*) - 1)}{d(x^*)} \mathbb{I}(r(v) = b) \right] \\
 &= \mathbb{I}(r \in a(r^*)) 
 \frac{n_1(r, x^*) - (d(x^*) - 1)}{d(x^*)} \mathbb{I}(r(v) = b). \\
\end{align*}

In this rejection case the state must have \(r(v) = b\), but it gets
weight proportional to \(n_1(r, x^*) - (d(x^*) - 1)\). The
\(\mathbb{I}(r(v) = b)\) part is easy to take care of by freezing node
\(v\) at color \(b\).

To understand the remaining expression \(n_1(r, x^*) - (d(x^*) - 1)\),
consider going through the nodes in \(U(x^*)\) in a specified order. As
each neighbor is encountered, it either is a color not seen by \(v\)
before that is blocked, or it does not block a new color. It might not
block a new color because it matches a frozen color, the forbidden
color, or a previous neighbor in \(U(x^*)\) that was already examined.

Altogether, the expression \(n_1(r, x^*) - (d(x^*) - 1)\) counts the
number of neighbors in \(U(x^*)\) that do \emph{not} block a new color
as they are considered.

The reasoning is as follows. The constant \(d(x^*) - 1\) represents the
number of colors that are possibly available once the frozen neighbor
colors are taken into account as well as the forbidden color that in the
case under consideration cannot be used by a frozen neighbor minus the
actual number of neighbors in \(U(x^*)\).

Each of these neighbors \(U(x^*)\) could have blocked a different color,
leaving exactly \(d(x^*) - 1\) colors for node \(v\) under \(x^*\).

But in fact, there were \(n_1(x, x^*)\) different colors for \(x^*\),
where \(n_1(x, x^*) \geq d(x^*) - 1\). So each neighbor that contributes
to \(n_1(x, x^*) - (d(x^*) - 1)\) is a neighbor that did \emph{not}
block a new color, but instead repeated a color already blocked.

For instance, suppose \(k = 10\), \(b = 10\), there is one neighbor
frozen at color \(3\), and the four unfrozen and non-ignored neighbors
are colors \(3, 2, 2, 1\) respectively.

Then the neighbor colored \(3\) does not block a new color, since that
was already blocked by the frozen neighbor. The first neighbor colored 2
blocks a new color, but the second neighbor colored 2 does not. Finally,
the neighbor colored 1 blocks a new color. So altogether there are two
neighbors that block new colors.

Now examine the expression for this example. Here
\(n_1(x, x^*) = 10 - 3 - 1 = 6\) as colors 4 through 9 are available to
be used for \(x(v)\). Next, there are 9 colors not blocked by frozen
neighbors and there are \(4\) neighbors that are neither frozen nor
ignored, making \(d(x^*) = 9 - 4 = 5\). Hence
\(n_1(r, x^*) - (d(x^*) - 1) = 6 - (5 - 1) = 6 - 4 = 2\).

The next step is to assign to each neighbor in \(U(x^*)\) a share of the
responsibility for this formula. A neighbor colored the forbidden color
receives share \(1\), as does a neighbor colored the same as a frozen
neighbor.

Suppose three neighbors have color \(4\). Then their contribution to
this formula is \(3 - 1 = 2\). If each of the three nodes is assigned a
share of \(2 / 3\), then their total contribution is
\((2/3) + (2 / 3) + (2 / 3) = 2\).

In general neighbors that match other neighbors in \(U(x^*)\) but not
frozen or forbidden colors get a share of \((\ell - 1) / \ell\), where
\(\ell\) is the number of neighbors in \(U(x^*)\) with the same color.
Let \(h(w)\) be the share for each node.

Choose from among the neighbors \(w\) of \(v\) with probability
proportional to \(h(w)\). Since \[
\sum_{w \in U(x^*)} h(w) = n_1(r, x^*) - (d(x^*) - 1),
\] the weight of the state when node \(w\) is frozen at its color is \[
w_3(r) = \mathbb{I}(r \in a(r^*)) 
 \frac{h(w)}{d(x^*)} \mathbb{I}(r(v) = b) \mathbb{I}(\text{neighbor } w \text{ frozen at } r(w)).
\]

Freezing this node reveals its color. If it is the forbidden color or
that of another frozen node, then \(h(w) = 1\) and the weights are
uniform.

If, however, this node is one of these colors, there must be at least 1
match and perhaps more among the rest of the neighbors. The best that
can be done is to freeze all members of \(U(x^*)\) that share this
color, then refreeze \(v\) to this color to indicate that no other
members of \(U(x^*)\) have this color, and finally forbid the neighbors
of \(v\) to the original color \(b\) because this the subcase that the
state is in.

Putting all this together, the procedure for removing a forbidden color
is as follows.

\vspace*{10pt}

\textbf{Remove forbidden}\((x, x^*, v)\)

\emph{Input} state \(x\) with index \(x^*\) where \(x^*(v) = -b\)

\begin{enumerate}
\def\labelenumi{\arabic{enumi}.}
\item
  If \((\exists t \in N(v))(x^*(t) = b)\), then let
  \(y(v) \leftarrow x(v)\),
  \(y^*(V \setminus \{ v \}) \leftarrow x^*(V \setminus \{ v \})\),
  \(y^*(v) \leftarrow \bot\), output \((y, y^*)\) and exit.
\item
  Let \(r(V \setminus \{v\}) \leftarrow x(V \setminus \{v\})\),
  \(r^*(V \setminus \{ v \}) \leftarrow x^*(V \setminus \{ v \})\),
  \(r^*(v) \leftarrow \bot\).
\item
  Let \(d\) be the number of colors not used by frozen neighbors of
  \(v\) in \(x^*\) minus the number of neighbors that are not frozen or
  ignored. With probability \(1 / d\) let \(r(v) \leftarrow b\),
  otherwise let \(r(v) \leftarrow x(v)\).
\item
  If \(r(v) \neq b\) then output \((r, r^*)\) and exit.
\item
  If there is a neighbor of \(v\) with color \(b\), search uniformly
  without replacement through the neighbors of \(v\) until such a
  neighbor \(w\) is found. Set \(s^*\) to be the index that is the same
  as \(r^*\) except the neighbors searched before \(w\) are forbidden to
  have color \(b\) and \(w\) is frozen at \(b\). Output \((r, s^*)\) and
  exit.
\item
  Let \(n_1\) be the number of colors valid for node \(v\) using
  \(x(V \setminus \{v\})\) under \(x^*\). With probability
  \((d - 1) / n_1\), output \((r, r^*)\) and exit.
\item
  For every unfrozen or non-ignored neighbor \(w\) of \(v\), set
  \(h(w)\) to be 1 if it shares a color with a frozen neighbor of \(v\)
  or is the forbidden color, or \((g - 1) / g\) if \(w\) is part of a
  group of \(g\) unfrozen or non-ignored neighbors of \(v\) that all
  have the same color. Choose a node \(w\) with probability proportional
  to \(h(w)\).
\item
  If \(w\) has the same color as a frozen node or is the forbidden
  color, let \(q^*\) be the same index as \(r^*\) but with \(v\) and
  \(w\) frozen at their colors, output \((r, q^*)\), and exit.
\item
  Let \((t, t^*)\) be the same state-index pair as \((r, r^*)\) except
  all the neighbors of \(v\) with the same color as \(w\) are frozen at
  their color, any remaining unfrozen neighbors of \(v\) are forbidden
  to have color \(b\), and \(v\) is frozen at the same color as \(w\).
  Output \((t, t^*)\) and exit.
\end{enumerate}

\begin{lemma}
\textbf{Remove forbidden} has the weight function given by the output
index for the output state.

\end{lemma}

\begin{proof}
The proof (while more complicated) follows along the similar lines as
the previous ones. For each possible outcome index, it is necessary to
check that the weight ends up being the proper one.

When the output occurs in line 1 of the step, then \(x = y\) and
\(a(x^*) = a(y^*)\) so \(w_{x^*}(x) = w_{y^*}(y)\) and there is no need
to consider further.

When the output has \(y^* = r^*\) in line 4 or line 6, recall that
\(p_1\) and \(p_2\) were both chosen to ensure that \(w_{y^*}(y)\) was
correct.

When the output has \(y^* = s^*\) at line 5, the search for the blocking
neighbor has the same probability as in the proof of the previous fact
and the calculation is identical here.

When the output is \((r, q^*)\) in line 8, then \(h(w) = 1\) since from
\(q^*\) the newly frozen neighbor \(w\) of \(v\) either matches the
color of another frozen neighbor of \(v\) or the forbidden color.

The index \(q^*\) freezes \(v\) at color \(b\). Hence the weight factor
\(\mathbb{I}(r(v) = b)\) in \(w_3\) becomes a function \(f_1\) of
\(q^*\) and so can be written as \[
w_{q^*}(r) = \mathbb{I}(r \in a(q^*)) 
 \frac{1}{d(x^*)} f(q^*).
\] This makes the distribution uniform over \(a(q^*)\).

In the last output type \((t, t^*)\) at line 9, since every node with
the same color as \(w\) is frozen in \(t^*\), it is easy to count the
number of neighbors of \(v\) with that color. Hence \(h(w) = f_2(t^*)\)
for a function \(f_2\). Moreover, since all the other unfrozen neighbors
of \(v\) are forbidden color \(b\), \(\mathbb{I}(r(v) = b)\) is a
function \(f_3(t^*)\). This makes \(w_3\) equal to \[
w_{t^*}(t) = \mathbb{I}(t \in a(t^*)) 
 \frac{f_2(t^*)}{d(x^*)} f_3(t^*)
\] and the final distribution is uniform over \(a(t^*)\).

\end{proof}

\section{Proving Theorem 1}
\label{proving-theorem-1}

Each step in the main algorithm was written by saying select \(v\) such
that \(x^*(v)\) has some property, but it is straightforward to write
the index as a four component partition of the set of nodes \(V\)
labeled
\((V_{\text{for}}, V_{\text{froz}}, V_{\text{ig}}, V_{\text{unr}})\)
together with a single color \(b\) that all the nodes in
\(V_{\text{for}}\) are forbidden to have. This makes selection of \(v\)
a constant time process.

Each step then makes a fixed number of changes to \(v\), its neighbors
and possibly the neighbors of a single neighbor. For \(\Delta\) the
maximum degree of the graph, this means that every step takes time at
most \(O(\Delta)\) to execute.

Aside from various coin flips undertaken at each step, some steps also
choose a color uniformly from \(C\).   Also, \( O(\Delta \log(\Delta)) \) bits are used to search the non-ignored unfrozen neighbors of \( v \) uniformly at random without replacement.  Therefore \(O(\Delta \log(k))\) uniform
random bits suffices to make the random draws in a single step.

So the remaining question is, starting from the state where all nodes
are ignored, what is the expected number of steps needed to make a node
unrestricted?

If the current index has every node ignored, then the expected number of
steps in \textbf{RR algorithm for proper colorings} needed to reach the
index where every node is unrestricted is at most \(O(\#V)\), where each
step takes \(O(\Delta)\) time to implement when \(k > 3.637 \Delta\).

Let \(T\) denote the number of steps needed for the algorithm to
terminate.

As usual with these types of proofs, a potential function \(\phi\) on
indices will be used that is shown to decrease by \(\epsilon > 0\) at
each step. That is, for \(\mathcal{A}(x, x^*)\) equal to \(y^*\) for any
of \textbf{Remove frozen}, \textbf{Remove forbidden}, or \textbf{Remove
ignored} with \(v\) any appropriate node, it will be shown that \[
\mathbb{E}[\phi(\mathcal{A}(x, x^*)) | \phi(x^*)] \leq \phi(x^*) - \epsilon.
\]

Standard result from martingale theory (together with the fact that
there is positive probability of \(\phi(x^*)\) decreasing by a fixed
amount each step) then give that \[
\mathbb{E}[T] \leq \phi(x_0^*) / \epsilon.
\]

This function assigns each condition (forbidden, frozen, ignored) a
positive weight. Since the weights can be multiplied by any positive
factor and still work, the weight of an ignored node can be scaled to be
1. \begin{align*}
\phi(x^*) &= w_1 \#(\{v:x^*(v) \in C \}) + w_2 \#(\{v:x^*(v) \in - C\}) + \#(\{v:x^*(v) = 0\}) \\
 &= w_1 \#V_{\text{froz}} + w_2 \#V_{\text{for}} + \#V_{\text{ig}}.
\end{align*}

Therefore, for the initial state \(x^*_0\) where every node is ignored,
\(\phi(x^*_0) = \#V\). At termination, \(\phi(x^*) = 0\).

Now set the weights.

\textbf{Remove frozen.} The frozen removal allows us to understand how
to relate \(w_1\) and \(w_2\). Remember that removal of a frozen node
results in one ignored node and the possible addition of up to
\(\Delta\) forbidden nodes. Therefore, by setting \[
w_1 = 1 + \Delta w_2 + \epsilon,
\] the expected change from removing a frozen node will be at most
\(-\epsilon\).

Now consider the ignored node. If \(x(v)\) does not match any of its
neighbors, one ignored node is removed. But if it does, it creates at most on
average \((\Delta - 1) / 2\) forbidden nodes and one frozen node.

\textbf{Remove ignored}. Any ignored node has color uniformly chosen
over \(C\). Therefore the chance of a conflict is at most
\(\Delta / k\), so the expected change in the potential function is at
most \[
(-1)\frac{k - \Delta}{k} + \frac{\Delta}{k}\left[ \frac{\Delta - 1}{2} w_2 + w_1 \right]
\] Let \[
\alpha = \frac{k - 1}{\Delta},
\] then this is upper bounded by \[
-\frac{\alpha - 1}{\alpha} + \frac{1}{\alpha}\left[\frac{3\Delta - 1}{2} w_2 + \epsilon \right]
\] This is less than \(-\epsilon\) when \begin{align}
w_2 = \frac{2}{3\Delta - 1}(\alpha - 1 - \epsilon(\alpha + 1))
\end{align} for \(0 < \epsilon < (\alpha - 1) / (\alpha + 1)\).

\textbf{Remove forbidden}

Now consider what happens when an attempt is made to remove the
forbidden condition at \(b\). In line 1, the condition is simply
removed, so the expected change is \(- w_2\).

In line 4 the condition is also removed. This line activates with
probability \(1 - 1 / d\) and also reduces the number of forbidden
conditions by 1.

With probability \(1 / d\), an attempt is made to recolor \(v\) with the
forbidden color \(b\).

If a neighbor is already colored \(b\) this results in rejection. The
search for this neighbor, adds (on average) \(\Delta - 1\) forbidden
nodes, then \(w\) is frozen and \(v\) changes to ignored. On the bright
side, the forbidden condition on \(v\) is removed, so this average
change in the potential function is \[
\frac{1}{d}\left[-w_2 + \frac{\Delta - 1}{2}w_2 + 1 + (1 + \Delta w_2 + \epsilon)\right]
= \frac{1}{d} \left[ \frac{3(\Delta -1)}{2} w_2 + 2 + \epsilon \right].
\]

If there is no neighbor colored \(b\), then the chance of rejection is
\(1 / n_1(x, x^*) \leq 1 / (k - \Delta - 1) = (1 / (\alpha - 1))(1 / \Delta)\),
in which case \(v\) is frozen (but at least not forbidden any more.)

What is the expected number of unfrozen and non-ignored neighbors that
are frozen? There are at most \(\Delta\) of these neighbors. The chance
that a particular neighbor has the same color as one of the
\(\Delta - 1\) other neighbors is at most
\((\Delta - 1) / (k - \Delta - 1) \leq 1 / (\alpha - 1)\).

Lines 8 and 9 are reached only if the color of \(v\) is switched to
\(b\) which happens with probability \(1 / d\). In this case, when none
of the neighbors of \(v\) have color \(b\), on average how many
neighbors will be frozen?

Consider a particular neighbor \(w\) of \(v\). What is the chance that
it is frozen? Suppose the rest of the sample coloring is chosen except
for \(w\) using \(x^*\). Among the rest of the neighbors, there are
\(a_i\) with color \(i \in k\). Then the only way \(w\) can be frozen to
color \(i\) is if \(a_i > 0\), rejection occurs, \(w\) has color \(i\),
and color \(i\) is chosen. Note \(a_i > 0\) either because of a frozen
neighbor of \(v\) or an unfrozen neighbor.

The chance that \(w\) has color \(i\) is at most \[
\frac{1}{k - \Delta - 1}
\] since \(w\) has at most \(\Delta\) neighbors and there is possibly
one forbidden color for \(w\).

Let \(h\) be the sum over \(h(w')\) over the neighbors \(w'\) not
including \(w\). If \(w\) has color \(i\) where \(a_i > 0\), that
increases the sum of the \(h(w')\) by 1. Hence the probability of
picking color \(i\) to be frozen after choosing color \( i \) for \( w \) is
\[
\frac{h(w)}{h + 1}.
\]

The chance of rejection is \((h + 1) / n_1\), where \(n_1\) is the
number of colors available to node \(v\) under \(x^*\). Therefore, the
probability a neighbor \(w\) of \(v\) is frozen during lines 8 or 9 is
at most \[
\frac{1}{k - \Delta - 1} \cdot \frac{h(w)}{h + 1}\cdot \frac{h + 1}{n_1} = \frac{h(w)}{n_1(k - \Delta - 1)}.
\]

Summing these expected change bounds over all neighbors \(w'\) of \(v\)
gives an expected number of frozen neighbors that is at most \[
\sum_{w'} \frac{h(w')}{n_1(k - \Delta - 1)}  = \frac{h + 1}{n_1(k - \Delta - 1)} \leq 
\frac{\Delta}{(k - \Delta - 1)^2} = \frac{1}{(\alpha - 1)^2 \Delta}.
\]

Finally, if rejection occurs, \(v\) is frozen and might create up to
\(\Delta\) forbidden nodes, and so increases the number of frozen nodes
by 1. This brings the total change (including the \(1 / d\) chance of
trying to recolor \(v\) as \(b\) in the first place) to \[
\frac{1}{d} \left[\frac{h + 1}{n_1}\left[1 + \Delta w_2\right] + \frac{1}{(\alpha - 1)^2 \Delta}\right] \leq \frac{1}{d} \left[\frac{1}{\alpha - 1}\left[1 + \Delta w_2\right] +  \frac{1}{(\alpha -1)^2 \Delta} \right].
\]

This is less than the expected change of potential when there is a
neighbor of \(v\) colored \(b\), so that expression dominates.

Overall, an upper bound \(\delta\) on the expected change during this
step is \begin{align*}
\delta &\leq \left(1 - \frac{1}{d}\right)(-w_2) + 
 \frac{1}{d} \left[ \frac{3(\Delta - 1)}{2} w_2 + 2 + \epsilon \right].
\end{align*} To guarantee that \(\delta \leq - \epsilon\) requires \[
-\epsilon \geq -w_2 + \frac{1}{d}\left[ \frac{3(\Delta + 1)}{2} w_2 + 2 + \epsilon \right].
\] which is true if and only if \[
d \geq \frac{(3/2)(\Delta - 1) w_2 + 2 + \epsilon}{w_2 - \epsilon}.
\]

Now \[
d \geq k - \Delta \geq \Delta(\alpha - 1),
\] so this equation is satisfied if \[
\Delta(\alpha - 1) = \frac{(3/2)(\Delta - 1) w_2 + 2 + \epsilon}{w_2 - \epsilon}.
\]

Solving this equation gives \[
\epsilon
=
\frac{2\Delta \alpha^{2} - 7\Delta \alpha - \Delta + 3 \alpha - 1}
     {3\Delta^{2} \alpha - 3\Delta^2 + 2\Delta \alpha^2 - 4\Delta \alpha - \Delta + 3 \alpha + 2}.
\]

As long as \(\epsilon > 0\), the expected number of steps will be
bounded above by \(n / \epsilon\).

In the expression, to be greater than 0 for rising \(\Delta\) requires
that \(2\alpha^2 - 7\alpha - 1 > 0\), or
\(\alpha > (1 / 4)(7 + \sqrt{57}) = 3.637\ldots\).

This completes the proof of Theorem 1.

\bibliographystyle{plain}

\end{document}